\documentclass[12pt,reqno]{amsart}
\usepackage[english]{babel}
\usepackage{amssymb,amsmath,amsthm,latexsym}
\usepackage{amstext}
\usepackage{amsrefs}
\usepackage{amscd}
\usepackage{enumitem}
\usepackage{graphicx,epic,color,tikz}
\usepackage{amsfonts}
\usepackage{mathrsfs}
\usepackage[all]{xy}
\usepackage[enableskew,vcentermath]{youngtab}
\usepackage{ytableau}
\usepackage{inputenc}
\usepackage{multicol,hyperref}
\usepackage{faktor,stmaryrd,tikz-cd}
\usetikzlibrary{calc}
\usetikzlibrary{cd}
\allowdisplaybreaks

\numberwithin{equation}{section}

\theoremstyle{definition}
\theoremstyle{plain}
\newtheorem{definition}{Definition}[section]

\newtheorem{theorem}[definition]{Theorem}
\newtheorem{proposition}[definition]{Proposition}

\newtheorem{lemma}[definition]{Lemma}

\newcommand{\cc}{\mathbf{C}}
\newcommand{\zz}{\mathbf{Z}}
\newcommand{\rr}{\mathbf{R}}

\newcommand{\h}{\mathfrak{h}}

\newcommand{\sij}{s_{ij}}

\DeclareMathOperator{\Hom}{\mathrm{Hom}}

\DeclareMathOperator{\Ind}{\mathrm{Ind}}
\DeclareMathOperator{\Res}{\mathrm{Res}}

\DeclareMathOperator{\syt}{\mathrm{SYT}}
\newcommand{\bfmu}{\boldsymbol{\mu}}
\newcommand{\la}{\langle}
\newcommand{\ra}{\rangle}

\newcommand{\fu}{\mathfrak{u}}

\newcommand{\ct}{{\rm ct}}

\usepackage{todonotes}

\begin{document}
%%% OUTLINE %%%%

%\begin{enumerate}[label=\arabic*.,itemsep=.2in]
%\item Introduction
% 	\begin{enumerate}
%		\item statement of main theorem
%	(and definitions necessary to understand it)
%		\item Context: Previous or related work
%	(why is the result useful?)
%	\end{enumerate}
%
%\item Definitions and preliminaries (standard)
%
%\item Proof of main theorem(s) (Write this first)
%
%\item Examples (if any)
%\end{enumerate}
%\newpage
%%% PAPER %%%

\title{Classification of irreducible $\fu$-diagonalizable $H_{\ell,n}$-modules}
\author{Elizabeth Manosalva P.}
\thanks{This work was supported by ANID-FONDECYT, Chile Postdoctoral Grant \#3230329}

\maketitle
\begin{abstract}
We give a classification for the irreducible $\fu$-diagonalizable representations of the degenerate affine Hecke algebra of type $G(\ell,1,n)$. Precisely we show that such $H_{\ell,n}$-modules are indexed by $\ell$-skew shapes and that the representation indexed by a skew shape $D$ has a basis of eigenvectors indexed by standard Young tableaux of shape $D$.
	
\end{abstract}

\section{Introduction}

\subsection{Context} 
This article is contribution to the study of the representations of the affine Hecke algebras, in particular of the degenerate affine Hecke algebra of type $G(\ell,1,n)$, also known as graded Hecke algebra for the group $G(\ell,1,n)$.

 The affine Hecke algebra arise in several combinatoric and algebraic frameworks, since is closely related  to the geometry and representation theory of a semisimple Lie group and plays an important role in the Springer correspondence and the Langlands classification. Moreover Lusztig's work \cite{Lusztig2} shows that it is possible to recover geometric information contained in the affine Hecke algebra from the corresponding Hecke algebra. We also recall that the theory of intertwining operators originates from the study of affine Hecke algebra representations \cite{MatsumotoGraded} and it has played an important role in the theory of orthogonal polynomials, in particular the study of Macdonald polynomials \cite{CherednikGraded}, \cite{KSGraded}, \cite{OpdamGraded}.

More recently this algebra has played an important role in the study of the unitary representations of the rational Cherednik algebra of type $G(\ell,1,n)$ (or cyclotomic rational Cherednik algebra) thus giving the motivation for this work. The cyclotomic rational Cherednik algebra possesses a subalgebra which is a $G(\ell,1,n)$-analog of the degenerate affine Hecke algebra. Hence the classification theorem of the irreducible diagonalizable representations of $H_{\ell,n}$ given in this article together with results due to Griffeth \cite{SG-UnitaryReps}, Ciubotaru \cite{Ciubotaru} and Huang-Wong \cite{HuangWong} leads us to the combinatorial formula for the $V^*$-homology of the unitary representations of the cyclotomic rational Cherednik algebra \cite{FGM}. As the part of the proof having to do with the representation theory of $H_{\ell,n}$ was only sketched in \cite{FGM} this work comes to complement said article by delivering the detailed results in that matter.

%Since the cyclotomic rational Cherednik algebra possesses a subalgebra which is a $W$-analog of the degenerate affine Hecke algebra the classification theorem of the irreducible diagonalizable representations of $H_{\ell,n}$ given in this article together with the classification and graded dimensions of the irreducible unitary representations of the cyclotomic rational Cherednik algebra in category $\Oc$ obtained by Griffeth in \cite{SG-UnitaryReps} and results of Ciubotaru \cite{Ciubotaru} and Huang-Wong \cite{HuangWong}, we were able to obtain  in \cite{FGM} a non-negative combinatorial formula for the $V^*$-homology of the unitary representations of the cyclotomic rational Cherednik algebra.

The techniques used to obtain the results in this paper are mainly combinatorial, taking advantage of the combinatorial behavior of the group $G(\ell,1,n)$ and its representation as well as the affine Hecke algebra attached to $G(\ell,1,n)$. This work is intended to be mostly self-contained and states most of the needed definitions and results used in it. Nonetheless some results involving the cyclotomic rational Cherednik algebra may be needed and for which we refer  to \cite{EtingofGinzburg} and \cite{EtingofMaLectures}, for instance. 

\subsection{Statement of results}

Here we will state the main theorem of this work as well as the objects involved somewhat more precisely, referring to the body of the paper when appropriate. A detailed explanation of the combinatorial objects involved in the statement of Theorem \ref{MainThm} are given in further sections. For the moment, as a guide to the reader we mention:
 Let $H_{\ell,n}$ be the degenerate affine Hecke  algebra (\ref{SubsectionHln}). Let $\fu$ be a subalgebra of $H_{\ell,n}$ with $u_i,\zeta_i\in\fu$ (\ref{SubsectionSubalgebraU}). Given $M$ an $H_{\ell,n}$-module, $M$ is called $\fu$-diagonalizable if it has a basis consisting of simultaneous eigenvectors for $\fu$.
  An $\ell$-skew shape $D$ is a sequence $D=(D^0,D^1,\ldots D^{\ell-1})$ of $\ell$ skew diagrams (\ref{SubsectionSkewShapes}). For a box $b$ of $D$ in the component $D^j$ we write $\ct(b)=x-y$ and $\beta(b)=j$ for the \emph{content} and \emph{coordinate} of $b$, respectively. A standard Young tableau $T$ of $D$ is a filling of the boxes of $D$ with positive integer numbers which is left to right and top to bottom increasing in each coordinate $D^i$ of $D$ and each connected component. These ingredients given, we may state the theorem.

\begin{theorem}\label{MainThm}
	Let $M$ be an irreducible $\fu$-diagonalizable $H_{\ell,n}$-module and $m\in M$ satisfies $$u_im=a_im\quad\hbox{and}\quad \zeta_im=\zeta^{b_i}m\quad\hbox{for }1\leq i\leq n.$$ Then there is a standard Young tableau $T$ on a $\ell$-skew shape $D$ such that $a_i=\ell\ct(T^{-1}(i))$ and $b_i=\beta(T^{-1}(i))$ for $1\leq i\leq n$ and $M\cong S^D$, and moreover $T$ and $D$ are unique up to diagonal slides of their connected components.
\end{theorem}

This theorem gives a classification of the irreducible $\fu$-diagonalizable $H_{\ell,n}$-modules, we provide an algorithm to determine explicitly the diagram $D$ which represents the corresponding equivalence class in the proof of this theorem. We recall that the case $\ell=1$ was first worked out by Cherednik \cite{CherednikRepsAHalg}.

\section{Definitions and preliminaries}

This section is mostly dedicated to give in detail the needed definitions, notations and state results in order to prove the main theorem of this paper.

\subsection{Combinatorics.}\label{SubsectionCombinatorics}We recall some definitions and basic notions of partitions and standard young tableaux used in further sections.

A \emph{partition} is a non-increasing sequence $\lambda_1\geq\lambda_2\geq\ldots\geq \lambda_r\geq 0$ of positive integers. We identify a partition with its Young diagram and visualize it as a collection of boxes. The \emph{diagram} of $\lambda$ is the collection of points $(x,y)$ for $x$ and $y$ integer numbers satisfying $1\leq y\leq \ell$ and $1\leq x\leq \lambda_y$, which we think of as a subset of $\rr^2$. 

For a fixed integer $\ell>0$, an \textit{$\ell$-partition} is a sequence $\lambda=(\lambda^0,\lambda^1,\ldots,\lambda^{\ell-1})$ of partitions, some of which may be empty. If there are $n$ total boxes in $\lambda$ then we say that $\lambda$ is an \emph{$\ell$-partition of $n$} and we write $P_{\ell,n}$ for the set of all $\ell$-partitions of $n$.

We order the boxes of each $\ell$-partition as follows, let $b=(x,y)$ and $b'=(x',y')$ then 
\begin{equation*}
	b\leq b' \Leftrightarrow \beta(b)=\beta(b') \hbox{ with } x'-x\in\zz_{\geq 0}\hbox{ and } y'-y\in\zz_{\geq 0}
\end{equation*}

Let $\lambda$ be an $\ell$-partition of $n$, we define a \textit{standard Young tableau} of $\lambda$ as the filling of the boxes of $\lambda$ with the numbers $1,2,\ldots,n$ which is left to right and top to bottom increasing in each coordinate and think of it as a bijection between the boxes of $\lambda$ and the set $\{1,2,\ldots,n\}$. We denote by SYT($\lambda$) to the set of all standard Young tableau of $\lambda$.
For each box $b=(x,y)$ of the partition $\lambda^k$ of an $\ell$-partition $\lambda$  we define the \textit{content} and \textit{coordinate} of $b$ as $\ct(b):=x-y$ and $\beta(b):=k$, respectively.

\subsection{The group $G(\ell,1,n)$}\label{GroupGl1n}

By the classification and notation of Shephard and Todd \cite{ShephardTodd-URG} for complex reflection groups we denote by $G(\ell,1,n)$ to the group abstractly defined as the wreath product of the cyclic group of $\ell$ elements $\bfmu_\ell$ and the symmetric group $S_n$, which acts on $\h=\cc^n$ in the obvious way. Let $y_1,\ldots,y_n$ be the standard basis of $\h$ then the group $G(\ell,1,n)$ consists of all matrices of size $n$ by $n$ with exactly one non-zero entry in each row and column, which is an $\ell$th  root of unity. 

Note that with this characterization, the group $G(1,1,n)$ is the group of permutation matrices of size $n$ by $n$, which is isomorphic to the symmetric group $S_n$. And the group $G(2,1,n)$ is the group of signed permutation matrices also known as the Weyl group of type $B_n$.

For a fixed primitive $\ell$th root of unity $\zeta$, we denote by $\zeta_i$ to the diagonal matrix with ones all over the diagonal except in the position $i$ where there is $\zeta$, $\sij$ denotes the permutation matrix which exchanges coordinates $i$ and $j$, and $s_i=s_{i,i+1}$. With this notation we can write uniquely each element of $G(\ell,1,n)$ as $$\zeta^aw, \mbox{ with }w \in S_n, a\in\bfmu_\ell^n$$
where $\zeta^a=\zeta_1^{a_1}\cdots\zeta_n^{a_n}$.

The irreducible complex representations of $G(\ell,1,n)$ are in bijection with $\ell$-partitions $\lambda$ of $n$, and for $\lambda\in P_{\ell,n}$ we write $S^\lambda$ for the corresponding irreducible representation \cite{IrrG}. The $\cc G(\ell,1,n)$-module $S^\lambda$ possesses a basis indexed by the set SYT($\lambda$). Also we note that $$\dim S\lambda=n!\prod_{i=1}^\ell\prod_{b\in\lambda^i}\frac{1}{h_b}$$ where $h_b$ is the length of the hook with respect to the box $b$.

We write $\h=\cc^n$ for the defining representation, which is irreducible for $\ell>1$. Given $\lambda\in P_{\ell,n}$ we define its \emph{transpose} $\lambda^t$ as the $\ell$-partition obtained from $\lambda$ by cycling its components one spot to the left and transposing them all, i.e. 
\begin{equation*}
	\lambda^t=\big((\lambda^1)^t,(\lambda^2)^t,\ldots,(\lambda^{\ell-1})^t,(\lambda^0)^t \big)
\end{equation*}
where $\mu^t$ denotes the classical transpose of the partition $\mu$. 
With this indexing of the complex irreducible representations of $G(\ell,1,n)$, the representation indexed by $\lambda^t$ is the tensor product of the representation $\Lambda^n\h$.

\subsection{Jucys-Murphy elements.}\label{SubsectionJMelements} The following analogs of the Jucys-Murphy elements in the group algebra $\cc G(\ell,1,n)$ were introduced in \cite{IrrG}. In this section we mention some properties which will be useful in further sections. Define 

$$\phi_i=\sum_{\substack{1\leq j<i \\ 0\leq k\leq \ell-1}}\zeta_i^ks_{ij}\zeta_i^{-k}=\sum_{\substack{1\leq j<i \\ 0\leq k\leq \ell-1}}\zeta_j^ks_{ij}\zeta_j^{-k}\quad 1\leq i\leq n.$$

Besides the expected property for $\phi_i$ (which is to commute with every element of 
$\cc G(\ell,1,i-1)$), elements $\phi_1,\ldots,\phi_n$ are pairwise commutative and satisfy the following relations 
		\begin{eqnarray}
		\zeta_i\phi_j & = &\phi_j\zeta_i \mbox{ for } 1\leq i\leq n \mbox{ and } 1\leq j\leq n. \\
		 s_i\phi_j & = & \phi_js_i \mbox{ for } j\neq i,i+1. \\
		 s_i\phi_i  & = & \phi_{i+1}s_i-\pi_i \mbox{ for } 1\leq i\leq n-1\mbox{, where }\pi_i=\sum_{0\leq k\leq \ell-1}\zeta_i^k\zeta_{i+1}^{-k}\label{JMc}.
	\end{eqnarray}

\subsection{Subalgebra $\fu$}\label{SubsectionSubalgebraU}
Let $\fu$ be the subalgebra of $\cc G(\ell,1,n)$ generated by the Jucys-Murphy elements $\phi_1,\ldots,\phi_n$ and the diagonal matrices $\zeta_i,\ldots,\zeta_n$. The action of $\fu$ over a $\cc G(\ell,1,n)$-module $S^\mu$ is given by 
\begin{align*}
	\phi_iv_T & =\ell\ct(T^{-1}(i))(v_T)\\
	\zeta_iv_T & =\zeta^{B(T^{-1}(i))}v_T
\end{align*}
where the set $\{v_T|T\in SYT(\mu)\}$ is a basis for $S^\mu$.

Let $M$ be a $\fu$-module, for a given $\fu$-eigenvector $m\in M$ we define the \emph{weight} of the vector $m$ to be the tuple $$wt(m)=(a_1,\ldots,a_n,\zeta^{b_1},\ldots,\zeta^{b_n})$$ if $\phi_im=a_i m$ and $\zeta_im=\zeta^{b_i} m$ for $1\leq i\leq n$. 

\begin{lemma}\label{LemmaIntertwiners}
\begin{enumerate}
	\item The algebra $\fu$ acts semisimple on each $\cc G(\ell,1,n)$-module $M$.
	\item Let $M$ be a $\cc G(\ell,1,n)$-module and let $m\in M$ be a $\fu$-weight vector with $$wt(m)=(a_1,\ldots,a_n,\zeta^{b_1},\ldots,\zeta^{b_n})$$ then \begin{align*}
		(a_i,\zeta^{b_i})\neq (a_{i+1},\zeta^{b_{i+1}})\hbox{ for } 1\leq i\leq n-1
	\end{align*}
	\end{enumerate}
\end{lemma}
\begin{proof}
Let $M$ be a $\cc G(\ell,1,n)$-module, note that $\phi_i$ is a self adjoint operator and $\zeta_i$ is a unitary operator with respect to any $W$-invariant positive definite hermitian form on $M$. Then $\fu$ is a commutative algebra, which proves (a).\\

For (b) suppose that $(a_i,\zeta^{b_i})=(a_{i+1},\zeta^{b_{i+1}})$ for some $1\leq i\leq n-1$, then by the relation (\ref{JMc}) we have \begin{align*}
		\phi_is_i\cdot m = (s_i\phi_{i+1}-\pi_i)\cdot m = (s_ia_{i+1}-\pi)m=a_{i+1}s_im-\ell m
	\end{align*}
	hence $(\phi_i-a_i)s_im=-\ell m\neq 0$ while $(\phi_i-a_i)^2s_i m=-(\phi_i-a_i)\cdot \ell m=0$ so $si_i m$ is a generalized eigenvector, which is not an eigenvector for $\phi_i$ and hence contradicting (a).
\end{proof}

\subsection{The \emph{degenerate affine Hecke algebra} of type $G(\ell,1,n)$}\label{SubsectionHln} In the following section we define the \emph{degenerate affine Hecke algebra} which was introduced by A. Ram and V. Shepler in \cite{RamShepler} as the graded Hecke algebra for $G(\ell,1,n)$, it is also called \textit{generalized graded Hecke algebra} by C. Dez\'el\'ee in \cite{DezeleeHeckeAlg}. 

The \emph{degenerate affine Hecke algebra of type $G(\ell,1,n)$} is the algebra $H_{\ell,n}$ generated by $\cc[u_1,\ldots,u_n]$ and the group $G(\ell,1,n)$ subject to the relations
	\begin{align}
	\zeta_i u_j & = u_j\zeta_i,\quad \mbox{for all } i,j \label{RelHlnzetaiuj}\\
		s_iu_j & = u_js_i,\quad\mbox{if } j\neq i,i+1 \label{RelHlnsiuj} \\
		s_i u_i & = u_{i+1}s_i-\pi_i \label{RelHlnsiui}
	\end{align}
where $$\pi_i=\sum_{k=0}^{\ell-1}\zeta_i^k\zeta_{i+1}^{-k}.$$

Since $H_{\ell,n}$ can be injected onto the rational Cherednik algebra of type $G(\ell,1,n)$ the next two important results are consequences of the well-known PBW theorem for the rational Cherednik algebra (see \cite{EtingofGinzburg}). For detailed proof of this two statements we refer to \cite{DezeleeHeckeAlg}, \cite{MyThesis}.

\begin{theorem}[PBW]\label{DAHA_PBW} The set $\{u_{i_1}u_{i_2}\cdots u_{i_p}w|1\leq i_1\leq\cdots\leq i_p\leq n, w\in W\}$ is a $\cc$-basis for $H_{\ell,n}$.
\end{theorem}

\begin{proposition}
	The center of $H_{\ell,n}$ is $\cc[u_1,u_2,\ldots,u_n]^{S_n}\otimes \cc[\zeta_1,\zeta_2,\ldots,\zeta_n]^{S_n}$.
\end{proposition}

Note that from Theorem \ref{DAHA_PBW} we also have that $H_{\ell,n}$ is isomorphic to $\cc[u_1,u_2,\ldots,u_n]\otimes\cc G(\ell,1,n)$ as vector spaces.

\subsection{Intertwining operators}\label{SubsectionIntertwiningOp}

For $1\leq i\leq n$ define the intertwining operators $\tau_i$ by \begin{equation}\tau_i=s_i+\dfrac{c_0}{u_i-u_{i+1}}\pi_i.\end{equation}\label{taui}
Note that the operator $\tau_i$ is well-defined on $H_{\ell,n}$-modules since $u_i-u_{i+1}$ is invertible on the image of $\pi_i$.

The following proposition states some needed properties of the intertwining operators of $H_{\ell,n}$. As before this proof takes advantage of the injection of $H_{\ell,n}$ onto the rational Cherednik algebra nevertheless in this case we are able to give a sketch a proof since the results involved are not so deep.

\begin{proposition}\label{Properties Intertwiners}
	Let $M$ be a $H_{\ell,n}$-module and $m\in M$ such that $$wt(m)=(a_1,\ldots,a_n,\zeta^{b_1},\ldots,\zeta^{b_{n}}).$$
	\begin{enumerate}\itemsep.1in
		\item $wt(\tau_i m)=s_i(wt(m))$ where $S_n$ acts on the set of $2n$-tuples by simultaneously permuting the subindices of $a_i$ and $b_i$.\label{Prop Intertwiner Op a}
		\item $\tau_i^2=\dfrac{(u_i-u_{i+1}-\pi_i)(u_i-u_{i+1}+\pi_i)}{(u_i-u_{i+1})^2}$\label{Prop Intertwiner Op b}
		\item $\tau_i\tau_{i+1}\tau_i=\tau_{i+1}\tau_{i}\tau_{i+1}$\label{Prop Intertwiner Op c}
	\end{enumerate}
\end{proposition}

\begin{proof}[Sketch of proof]
To prove part \eqref{Prop Intertwiner Op a} we use \begin{equation}\label{u_itau_i}(u_i-u_{i+1})\tau_i=(u_i-u_{i+1})s_i+\pi_i\end{equation} (which is an element of the rational Cherednik algebra attached to $G(\ell,1,n)$) instead of \eqref{taui} and using the relations between $\tau_i$ and $u_i$ in this algebra we check by straightforward calculation that $\tau_im$ is an eigenvector with eigenvalue
\begin{align}\label{s_itau}
s_i((a_1,\ldots,a_n,\zeta^{b_1},\ldots,\zeta^{b_n})) & =(a_1,\ldots,a_{i+1},a_i,\ldots,a_n,\zeta^{b_1},\ldots,\zeta^{b_{i+1}},\zeta^{b_i},\ldots,\zeta^{b_n})
\end{align}

For \eqref{Prop Intertwiner Op b} we compute  \begin{equation}
	((u_i-u_{i+1})\tau_i)^2 = (u_i-u_{i+1})\tau_i(u_i-u_{i+1})\tau_i
\end{equation} 
then by \eqref{u_itau_i} and forward calculations we obtain the desired formula for $\tau_i^2$.

And part \eqref{Prop Intertwiner Op c} follows by using that $u_iu_{i+1}u_i=u_{i+1}u_iu_{i+1}$ by (a long) straightforward calculation.
\end{proof}

\subsection{$H_{\ell,n}$-modules via branching for $G(\ell,1,n)$}\label{SubsectionHlnMods}
Let $m\in\zz_{>0}$. Then for all $1\leq i\leq n$ and $1\leq j\leq n-1$ there exists a map $H_{\ell,n} \to\cc G(\ell,1,m+n)$ given by
\begin{align*}
	u_i & \mapsto \phi_{m+i} \\
	s_j & \mapsto s_{m+j} \\
	\zeta_i & \mapsto \zeta_{m+i}
\end{align*}

The image of this map is contained in the centralizer $C_{G(\ell,1,m+n)}G(\ell,1,m)$ so that $H_{\ell,n}$ acts on the module $$S^{\lambda\setminus\mu}=\Hom_{\cc G(\ell,1,m)}(S^\mu,\Res_m^{m+n}(S^\lambda))$$ for all pairs $\lambda,\mu$ of $\ell$-partitions, where $\lambda$ is an $\ell$-partition of $m+n$ and $\mu$ is an $\ell$-partition of $m$. It follows from Young's orthogonal form that $S^{\lambda\setminus\mu}=0$ unless $\mu\subseteq \lambda$ in which case $S^{\lambda\setminus\mu}$ has a basis indexed by the set of standard Young tableaux on the skew diagram $\lambda\setminus\mu$.

Given $T\in\syt(\mu)$ and $U\in\syt(\lambda\setminus\mu)$ we define $T\cup U\in\syt(\lambda)$ by 
$$T\cup U(b)=\begin{cases} T(b), & \hbox{if } b\in\mu \\
U(b)+m, & \hbox{if } b\in\lambda\setminus \mu
\end{cases}$$
Then $$S^{\lambda\setminus\mu}=\Hom_{\cc G(\ell,1,m)}(S^\mu,\Res_m^{m+n}(S^\lambda))$$ and we define $\psi_U\in S^{\lambda\setminus\mu}$ by the formula $$\psi_U(v_T)=v_{T\cup U}$$ for $v_T$ an element of the basis of $S^\mu$. Hence these $\psi_u$ are a basis of $S^{\lambda\setminus\mu}$, i.e. \begin{equation}\label{BasisHlnMod}
S^{\lambda\setminus\mu}=\cc\{\psi_u\,|\,U\in\syt(\lambda\setminus\mu)\} 	
 \end{equation}

It is easy to check that these $H_{\ell,n}$-modules are $\fu$-diagonalizable, therefore its irreducibility is proved by the main theorem. 

Since we aim to classify all the irreducible $\fu$-diagonalizable $H_{\ell,n}$-modules, given $M$ such a module and a $\fu$-eigenvector $m\in M$ we note that the weight $$wt(m)=(a_1,\ldots,a_n,\zeta^{b_1},\ldots,\zeta_{b_n})$$ could have coordinates  which are complex numbers instead of integers of even real ones, since $\phi_im=a_im$ and $\zeta_im=\zeta^{b_i}m$ for $1\leq i\leq n$. Then our next step in the classification is to generalize the definition of the module $S^{\lambda\setminus\mu}$ to $\ell$-skew diagrams with non-integer coordinates.

\subsection{Skew shapes}\label{SubsectionSkewShapes}
A \emph{skew shape} is a finite subset $D\subseteq\cc^2$ such that whenever $(x,y)\in D$ and $(x+k,y+l)\in D$ for nonnegative integers $k$ and $l$, then $(x+k',y+l')\in D$ for all integers $0\leq l'\leq l$ and $0\leq k'\leq k$. A skew shape $D$ is called \emph{integral} if $D\subseteq \zz_{>0}^2$. Each integral skew shape is the difference $D=\alpha\setminus\beta$ of two partitions, which are not uniquely determined by $D$. Boxes $b=(x,y)$ and $b'=(x',y')$ are called \emph{adjacent} if either
\begin{enumerate}
	\item $x=x'$ and $y=y'\pm 1$, or
	\item $y=y'$ and $x=x'\pm 1$.
\end{enumerate}

Adjacency in boxes of a skew shape is a equivalence relation, and the equivalent classes of a skew shape $D$ are called \emph{connected components} of $D$, and $D$ is \emph{connected} if it has only one connected component.

Like for $\ell$-partitions we have that an \emph{$\ell$-skew shape} is a sequence $D=(D^0,D^1,\ldots D^{\ell-1})$ of $\ell$ skew shapes, some of which may be empty. And for a box of $D$ in the component $D^j$, then we write $\ct(b)=x-y$ and $\beta(b)=j$ for the \emph{content} and \emph{coordinate} of $b$, respectively. Following the analogy with $\ell$-partitions, a standard Young tableau on $D$ is defined as a bijection $$T:\hbox{\{boxes of }D\}\to \{1,2,\ldots,n\}$$ which is left to right and top to bottom increasing in each coordinate $D^i$ of $D$ and in each connected component.

In this setup we also think of the points of a diagram as boxes, thus when a diagram $D\subseteq \zz^2$ we are in the case presented in (\ref{SubsectionHlnMods}).

\subsection{Automorphisms of $H_{\ell,n}$}\label{SubsectionAutomorphisms}

There are two types of automorphisms of $H_{\ell,n}$ that we will use to define the generalized version of the modules $S^{\lambda\setminus\mu}$.

\begin{proposition}
For each $\kappa\in\cc$ there exists unique automorphisms $t_\kappa$ and $\rho$ of $H_{\ell,n}$ given by
 \begin{align*}
 	t_\kappa (u_i) =u_i+\kappa, \quad t_\kappa(\zeta_i) = \zeta_i, \quad t_\kappa(s_i)=s_i \\
 	\rho(u_i)=-u_{n-i+1}, \quad  \rho(\zeta_i)=\zeta_{n-i+1} \quad \rho(s_i)=s_{n-i}.
 \end{align*}
\end{proposition}
\begin{proof}
 	We write $A=C\la u_1,\ldots,u_n,\zeta_1,\ldots,\zeta_n,s_1\ldots,s_{n-1}\ra$ for the tensor algebra on this generators, then $H_{\ell,n}\simeq A/I$ where 
 	\begin{align*}
 		I= \langle & u_iu_j-u_ju_i,\quad \hbox{for all } 1\leq i,j\leq n \\
 		& \zeta_j\zeta_i-\zeta_i\zeta_j, \quad \hbox{for all } 1\leq i,j\leq n \\
 		& s_is_j-s_js_i,\quad \hbox{for } j\neq i,i+1 \\
 		& s_is_{i+1}s_i-s_{i+1}s_is_{i+1} \\
 		& \zeta_iu_j-u_j\zeta_i,\quad \hbox{for all } 1\leq i,j\leq n \\
 		& s_iu_j-u_js_i\quad \hbox{for } j\neq i,i+1 \\
 		& s_iu_i-u_{i+1}s_i+\pi_i\rangle 
 	\end{align*} 
 
To check $t_\kappa$ is an automorphism, it is enough to check that $t_\kappa(a)\in I$, for all $a\in I$. We will check only the las relation. Note that $t_\kappa(\pi_i)=\pi_i$ since $t_\kappa(\zeta_i)=\zeta_i$, hence 
 \begin{align*}
 	 t_\kappa(s_i)t_\kappa(u_i)-t_\kappa(u_{i+1})t_\kappa(s_i)+t_\kappa(\pi_i) & =s_i(u_i+\kappa)-(u_{i+1}+\kappa)s_i+\pi_i \\
& = s_iu_i+\kappa s_i-u_{i+1}s_i-\kappa s_i+\pi_i \\
& = s_iu_i-u_{i+1}s_i+\pi_i \in I
 \end{align*}
In the same way we did for $t_\kappa$, we will check that $$\rho: A / I \to A / I$$ is an automorphism.
\begin{itemize}
	\item For the first relation \begin{align*}
		\rho(u_iu_j-u_ju_i) & =(-u_{n-i+1})(-u_{n-j+1})-(-u_{n-j+1})(-u_{n-i+1})\\
		& = u_{n-i+1}u_{n-j+1}-u_{n-j+1}u_{n-i+1} \in I
	\end{align*}
	\item For $1\leq i,j\leq n$ the cases $\rho(\zeta_i\zeta_j-\zeta_j\zeta_i),\rho(s_is_{i+1}s_i-s_{i+1}s_is_{i+1}),\rho(\zeta_iu_j-u_j\zeta_i)\in I$ and $\rho(s_is_j-s_js_i),\rho(s_iu_j-u_js_i)\in I$ for $j\neq i,i+1$ are similar.
	\item Note that \begin{align*}
		\rho(\pi_i)=\rho\left(\sum_{k=0}^{\ell-1}\zeta_i^k\zeta_{i+1}^{-k}\right) = \sum_{k=0}^{\ell-1}\zeta_{n-i+1}^{k}\zeta_{n-i}^{-k}=\pi_{n-i}
	\end{align*} hence \begin{align*}
		\rho(s_iu_i-u_{i+1}s_i+\pi_i) & = \rho(s_i)\rho(u_i)-\rho(u_{i+1})\rho(s_{i})+\rho(\pi_i) \\
		& = s_{n-i}(-u_{n-i+1}-(-u_{n-i})s_{n-i}+\pi_{n-i} \\
		& = u_{n-i}s_{n-i}-s_{n-i}u_{n-i+1}+\pi_i \\
		& = s_{n-i}(s_{n-i}u_{n-i}-u_{n-i+1}s_{n-i}+\pi_{n-i})s_{n-i}\in I
	\end{align*}
	then $s_{n-i}u_{n-i}-u_{n-i+1}s_{n-i}+\pi_{n-i}\in I$.
\end{itemize} 
\end{proof}

Both of these automorphisms preserve the group algebra $\cc G(\ell,1,n)$ and their restrictions to $\cc G(\ell,1,n)$ are inner.

Given an $H_{\ell,n}$-module $M$ and an automorphism $a$ of $H_{\ell,n}$ we write $M^a$ for the $H_{\ell,n}$-module which is equal to $M$ as an abelian group with the $H_{\ell,n}$-action defined by $$h\cdot m=a(h)m\quad\hbox{ for }h\in H_{\ell,n}\hbox{ and } m\in M.$$

If $a=\rho$ or $a=t_\kappa$ then $M^a$ is isomorphic to $M$ as a $\cc G(\ell,1,n)$-module, since the restrictions of this automorphisms to $G(\ell,1,n)$ are inner.

\subsection{$H_{\ell,n}$-modules $S^D$}\label{SubsectionHlnModsSD} Here, using the tools and known results for $S^{\lambda\setminus\mu}$ and the new objects and definitions given in (\ref{SubsectionSkewShapes}) and (\ref{SubsectionAutomorphisms}) we proceed to generalize these modules.

Let $D\subseteq\rr^2\times\zz/\ell  $ be a skew shape with connected components $D_1,\ldots,D_k$. After diagonal slides, we may assume that each $D_i$ is such that for any $(x,y)\in D_i$ we have $y_i\in\zz$ and moreover the set of $y$-coordinates of distinct $D_i$'s are disjoint. We may choose (non-unique) $\alpha_1,\ldots,\alpha_k\in\cc$ and integral skew shapes $\lambda_1\setminus\mu_1,\ldots,\lambda_k\setminus\mu_k$ such that
\begin{equation*}
	D_i=\lambda_i\setminus\mu_i+(\alpha_i,0)
\end{equation*}
so that their union is disjoint and a skew shape, $$\lambda\setminus\mu=\coprod\lambda_i\setminus\mu_i$$ and so $\lambda_1\setminus\mu_1,\ldots,\lambda_k\setminus\mu_k$ are the connected components of their disjoint union $\lambda\setminus\mu$.
Then define
\begin{equation*}
S^D=\Ind^{H_{\ell,n}}_A\big(\otimes_i(S^{\lambda_i\setminus\mu_i})^{t_{\alpha_i}}\big)
\end{equation*}
where $A=H_{\ell,n_1}\otimes H_{\ell,n_2}\otimes\cdots\otimes H_{\ell,n_k}$ and $n_1,\ldots,n_k$ are certain nonnegative integers such that $\sum {n_i}=n$.

We obtain by this construction an $H_{\ell,n}$-module $S^D$ which is isomorphic to $S^{\lambda\setminus\mu}$ as $G(\ell,1,n)$-modules. We will see in the following section that this is indeed the generalization we were looking for and moreover is a very good one since each $\fu$-diagonalizable $H_{\ell,n}$-module can be obtained as one of this.

\begin{theorem}\label{SD irreducible}
	$S^D$ is an irreducible $\fu$-diagonalizable $H_{\ell,n}$-module
\end{theorem}

In order to prove this result first we recall some combinatorial lemmas involving skew shaped diagrams and tableaux.

Let $D$ be a skew diagram of shape $\lambda\setminus\mu$, for a box $b\in D$ with $b=(x,y)$ we define $R(b)=y$ and $C(b)=x$, i.e. $R(b)$ and $C(b)$ be the row and the column of the box $b$ in $D$, respectively.
Let $T$ be a standard Young tableau in $D$, define $I(T)$ as the set of pairs $(i,j)$ of numbers $1\leq i<j\leq n$ such that $j$ appears in $D$ in a row strictly above than $i$, i.e.
		 $$I(T):=\{(i,j)\,|\, R(T^{-1}(i))>R(T^{-1}(j)), \hbox{ for } 1\leq i<j\leq n \}$$
		 
\begin{lemma}\label{CombLemma1}
	Suppose $(i,i+1)\in I(T)$ and $s_iT\in{\rm SYT}(D)$ then $I(s_iT)=s_i(I(T)\setminus\{(i,i+1)\})$.
\end{lemma}

\begin{lemma}\label{CombLemma2}
	If $(i,i+1)\notin I(T)$ for all $1\leq i\leq n$ then $T$ is the \emph{row reading tableau}.
\end{lemma}

Moreover note that if $T$ is the \emph{row reading tableau} in $D$ then $R(T^{-1}(i))\leq R(T^{-1}(j))$ for all $1\leq i<j\leq n-1$ then clearly $(i,j)\notin I(T)$ and $I(T)=\emptyset$.  If $\lambda$ is a row then there is only one standard Young tableau which is the row reading.

\begin{lemma}\label{CombLemma3}
	Let $D$ be a skew diagram with $n$ boxes. Given $T$ and $T'$ two standard young tableaux on $D$ then there exists a sequence of simple transpositions $s_{i_1},s_{i_2},\ldots,s_{i_p}$ such that $$s_{i_p}s_{i_{p-1}}\cdots s_{i_1}(T)=T'$$ and $s_{i_k}s_{i_{k-1}}\cdots s_{i_1}(T)\in\syt(D)$ for all $1\leq k\leq p$.
\end{lemma}

Note that this sequence of simple transpositions given in the preceding lemma corresponds to a sequence of invertible intertwining operators, and as a consequence of this lemma we can ``connect" every pair of standard Young tableaux of shape $D$ by a sequence of invertible intertwining operators.

\begin{proof}[Proof of theorem \ref{SD irreducible}]
Let $n_i$ be the number of boxes of the diagram $\lambda_i\setminus\mu_i$, and $n=\sum_i^k n_i$.

From (\ref{SubsectionHlnMods}) each $S^{\lambda_i\setminus\mu_i}$ has basis $\{\psi_U\,|\,U\in\syt(\lambda_i\setminus\mu_i)\}$. Since $\lambda\setminus \mu=\coprod \lambda_i\setminus\mu_i$ then the set

\begin{equation}\label{Basis for H l n1 to nk}
	\Bigl\{\psi_U\,\Bigl|\,U\in\syt(\lambda\setminus\mu)\hbox{ such that } U(b)\in\{n_{i-1}+1,\ldots,n_{i}\}\hbox{ if }b\in\lambda^i\setminus\mu^i \Bigl\}
\end{equation} is a $\cc$ basis of $\bigotimes_i S^{\lambda_i\setminus\mu_i}$, where $n_0=0$.

Note that $H_{\ell,n}$ can be understand as the free module over the subalgebra $H_{\ell,n_1}\otimes H_{\ell,n_2}\otimes\cdots\otimes H_{\ell,n_k}$ which by PBW-theorem for $H_{\ell,n}$ has basis indexed by the set
\begin{equation}\label{Basis2}
\Bigl\{w\,\Bigl|w\in S_n/S_{n_1}\times S_{n_2}\times \cdots \times S_{n_k}\Bigl\}
\end{equation}

Then the $H_{\ell,n}$-module
$$S^D=\Ind_{H_{\ell,n_1}\otimes H_{\ell,n_2}\otimes\cdots\otimes H_{\ell,n_k}}^{H_{\ell,n}}\Bigl(\bigotimes_i S^{\lambda_i\setminus\mu_i}\Bigl)= H_{\ell,n}\otimes_{H_{\ell,n_1}\otimes H_{\ell,n_2}\otimes\cdots\otimes H_{\ell,n_k}} \Bigl(\bigotimes_i S^{\lambda_i\setminus\mu_i}\Bigl)^{t_{\alpha_i}}$$
 has a basis $\psi_U \otimes w$ where $\psi_U$ runs over the set \eqref{Basis for H l n1 to nk} and $w$ runs over the set \eqref{Basis2}, these pairs are in bijection with standard Young tableaux of shape $\lambda\setminus\mu$.
 
 Let $U$ as in \eqref{Basis for H l n1 to nk} and $w$ as in \eqref{Basis2} with minimal length (as a word on simple transpositions). For a fixed reduced expression for $w=s_{i_1}\cdots s_{i_p}$ the correspondence given by $$s_{i_1}\cdots s_{i_p}\to \tau_{i_1}\cdots\tau_{i_p}$$ is well defined, since the intertwining operators $\tau_i$ satisfy the braid relations (Proposition \ref{Properties Intertwiners}(\ref{Prop Intertwiner Op c})). 
 Note that $wU\in\syt(\lambda\setminus\mu)$ and
\begin{equation}
 \tau_{i_1}\cdots \tau_{i_p}(1\otimes \psi_U) = w \otimes \psi_U+\sum v\otimes \psi_U	
\end{equation}
 is an eigenvector with eigenvalue $wU$, where the length of the element $v$ is lower than the length of $w$. This gives rise to a basis for $S^D$ of eigenvectors of $\fu$, now if $S^D$ has a nonzero submodule, such submodule must contain one of this eigenvectors and by Lemma \ref{CombLemma3} we can connect a pair of two Young tableaux by a sequence of invertible intertwining operators, hence $S^D$ is irreducible.
\end{proof}

Note that from Young's orthogonal form for skew shapes \cite{GriffethNorton} we have that up to isomorphism the representation $S^D$ is independent of the choices made in this construction. Also since the automorphisms $t_\kappa$ are the identity in $G(\ell,1,n)$ it follows that its restriction to $G(\ell,1,n)$ is isomorphic to $S^{\lambda\setminus\mu}$.

\section{Proof of main theorem}\label{MainTheoremSection}

In order to prove the main result of our paper we begin by proving the following lemma. We note in advance that the proof of the lemma amounts to an effective recursion for constructing $T$ and $D$.

\begin{lemma}\label{LemmaMain}
	Let $(a_1,\ldots,a_n,\zeta^{b_1},\ldots,\zeta^{b_n})$ be a sequence satisfying the property: if $i<j$ with $a_i=a_j$ and $b_i=b_j \mod\ell$, then there are $i<k,m<j$ with $b_k=b_m=b_i\mod\ell$ and $$a_k=a_i+\ell,\quad a_m=a_i-\ell.$$
	Then there is a skew shape $D$ and a standard Young tableau $T$ of shape $D$ satisfying $$\ell \ct (T^{-1}(i))=a_i\quad\hbox{and}\quad \beta(T^{-1}(i))=b_i,\quad\hbox{ for } 1\leq i\leq n.$$ Moreover, $T$ and $D$ are unique up to diagonal slides of their connected components.
\end{lemma}
\begin{proof}
We proceed by induction on $n$. Let $\ell=1$ then we have the sequence $(a_1,\ldots,a_n)$.  In this case $D$ has only one component, if $n=1$ then $D$ has a single box with content $a_1$ and $T$ is the standard Young tableau that assigns to that box the number 1.
 For $n>1$,  by induction we suppose that the sequence $(a_1,\ldots,a_{n-1})$ possesses a standard Young tableau $T'$ on a skew diagram $D'$. Notice that the condition: ``if $i<j$ with $a_i=a_j$ then there are $i<k,m<j$ with $a_k=a_i+\ell$ and $a_m=a_i-\ell$" implies that for boxes of $D'$ such that $\ct(T^{-1}(k))=a_i+1$ and $\ct(T^{-1}(m))=a_i-1$ we have $$\ct(T^{-1}(k))=\ct(T^{-1}(i))+1 \quad\hbox{and}\quad\ct(T^{-1}(m))=\ct(T^{-1}(i))-1.$$  Then the condition implies that $D'$ is in fact a skew diagram and that $D'$ possesses an addable box $b$ with $\ct(b)=a_n$. We obtain $T$ and $D$ by adjoining $b$ to $D'$ and defining $T(b)=n$.
 
 For the case $\ell\neq 1$ then $D$ has $\ell$ components. We obtain the component $D_{i_1}$ by considering the subsequence $(a_{i_1},\ldots,a_{i_p})$ of $(a_1,\ldots,a_n)$ such that $b_{i_1}=\cdots=b_{i_p}\mod \ell$ and proceeding as before.
\end{proof}

\setcounter{section}{1}
\setcounter{definition}{0}
\begin{theorem}
	Let $M$ be an irreducible $\fu$-diagonalizable $H_{\ell,n}$-module and $m\in M$ satisfies $$u_im=a_im\quad\hbox{and}\quad \zeta_im=\zeta^{b_i}m\quad\hbox{for }1\leq i\leq n.$$ Then there is a standard Young tableau $T$ on a skew shape $D$ such that $a_i=\ell\ct(T^{-1}(i))$ and $b_i=\beta(T^{-1}(i))$ for $1\leq i\leq n$ and $M\cong S^D$, and moreover $T$ and $D$ are unique up to diagonal slides of their connected components.
\end{theorem}
\setcounter{section}{3}
\setcounter{definition}{6}

\begin{proof} For the first part we will check that the sequence $(a_1,\ldots,a_n,\zeta^{b_1},\ldots,\zeta^{b_n})$ satisfies the hypothesis of Lemma \ref{LemmaMain}. Suppose that there is an index $j>i$ with $a_i=a_j$ and $b_i=b_j \mod \ell$, we claim that $j>i+2$. 

Suppose that $j=i+1$, then $$(a_i,\zeta^{b_i})=(a_{i+1},\zeta^{b_{i+1}}),$$ and by the relations between $u_i$ and $s_i$ we have
\begin{align}
	u_is_im & = (s_iu_{i+1}-\pi_i)m = s_iu_{i+1}m-\pi_i m= s_ia_im-\ell m = a_is_im-\ell m \label{uisi}\\
	u_{i+1}s_im & = (s_iu_i+\pi_i)m = s_iu_im+\pi_im = s_ia_im+\ell m = a_is_im+\ell m \label{ui+1si}
\end{align}

Notice that $s_im=\pm m$, since $s_i^2=1$. Then if we suppose $s_im=m$ then $u_is_im=u_im=a_im$ and if $s_im=-m$ then $u_is_im=-a_im$, both cases leads to a contradiction with \eqref{uisi}. We conclude that $s_im$ and $m$ are linearly independent and the subspace of $M$ generated by $m$ and $s_im$ is stable under $u_i$ and $u_{i+1}$. Moreover $u_i=\begin{pmatrix} a_i & -\ell \\ 0 & a_i
\end{pmatrix}$ and $u_{i+1}=\begin{pmatrix}
	a_i & \ell \\ 0 & a_i
\end{pmatrix}$, then $M$ is not diagonalizable since there exist this Jordan blocks, which is a contradiction and then $j\neq i+1$.

For $j=i+2$, note that if $\tau_i$ was invertible then by part (a) of Proposition \ref{Properties Intertwiners} $\tau_i(m)$ has eigenvalue $$(a_1,\ldots,a_{i+1},a_i,a_{i+2},\ldots,a_n,\zeta^{b_1},\ldots,\zeta^{b_{i+1}},\zeta^{b_i},\zeta^{b_{i+2}},\ldots,\zeta^{b_n})$$
with $a_i=a_{i+1}$ and $b_i=b_{i+2}$ which contradicts the previous case. Hence $\tau_i$ is not invertible, then necessarily $b_i=b_{i+1}$ and $a_{i+1}= a_i\pm \ell$.
If $a_{i+1}=a_i+\ell$ then $s_i$ is acting by 1  on $m$ and $s_{i+1}$ is acting by $-1$ on $m$, which contradicts the braid relation. Analogously if $a_{i+1}=a_i-\ell$, $s_i$ acts by $-1$ on $m$ and $s_{i+1}$ acts by $1$, contradicting the braid relation again.

Since $j>i+2$ and the operators $\tau_i$ and $\tau_j$ are invertible then by Lemma \ref{LemmaMain} there is $T$ and $D$ as required.

It remains to prove that $M\cong S^D$. Let $T_0$ be the row reading tableau and $\psi_{T_0}$ the corresponding eigenvector in $S^D$. Let $m\in M$ with $wt(T_0)=wt(m)$, we will show there exists a homomorphism $\varphi: S^D\to M$ such that $\varphi(\psi_{T_0})=m$. Given $T\in \syt(D)$ by Lemma \ref{CombLemma3} there exists a sequence of simple transpositions $s_{i_1},\ldots,s_{i_p}$ such that $s_{i_p}\cdots s_{i_1}T=T_0$ with $s_{i_k}s_{i_{k-1}}\cdots s_{i_1}T_0\in \syt(D)$, this produces a sequence of invertible intertwining operators $\tau_{i_1},\ldots,\tau_{i_p}$. Choosing $p$ minimal (as a word in the symmetric group) we define a map given by $$\varphi(\tau_{i_p}\ldots \tau_{i_1}\psi_{T})=\tau_{i_p}\cdots\tau_{i_1}m.$$
Firstly, this map is well-defined which can be proved by the solution of the word problem in the braid group since the intertwining operators $\tau_i$ satisfy the braid relations. Secondly, it is compatible with the action of $\fu$ since commutes with the action $\tau_i$. Thus we have produced a non-zero homomorphism between the irreducible $H_{\ell,n}$-modules $S^D$ and $M$, which finishes the proof.
\end{proof}

\bibliography{mibiblio}
\bibliographystyle{amsalpha}
\bigskip

{\footnotesize\textsc{
 Departamento de Matem\'atica, Universidad de Chile}}.
 
{\footnotesize \textit{ E-mail address:} \texttt{elymanosalva.ima@gmail.com}}

\end{document}